\documentclass[12pt]{amsart}

\usepackage{fullpage} 
\usepackage{parskip} 
\usepackage{tikz-cd} 
\usepackage{amsmath, amssymb, amsthm}
\usepackage{esint}
\usepackage{hyperref}
\hypersetup{
	colorlinks = false,
	linkbordercolor = {white},
	citebordercolor = {white},
}

\usepackage[utf8]{inputenc}
\usepackage[english]{babel}
\usepackage[alphabetic]{amsrefs}

\usepackage{graphicx}
\usepackage{float}

\usepackage[title]{appendix}

\usepackage{csquotes}

\usepackage{geometry}
\usepackage{fancyhdr}
\usepackage{lipsum}

\numberwithin{equation}{section}

\theoremstyle{definition}
\newtheorem{thm}[equation]{Theorem} 
\theoremstyle{definition}
\newtheorem{definition}[equation]{Definition}

\newtheorem{prop}[equation]{Proposition} 
\newtheorem{cor}[equation]{Corollary}
\newtheorem{lemma}[equation]{Lemma}

\newcommand{\ovl}[1]{\overline{#1}}
\newcommand{\pair}[1]{\left\langle #1\right\rangle}
\newcommand{\bb}[1]{\mathbb{#1}}

\newcommand{\V}{\text{Vol}}

\newcommand{\D}{\Delta}

\newcommand{\n}{\nabla}
\newcommand{\Set}{\mathcal S}

\newcommand{\sbst}{\subseteq}
\newcommand{\h}{\textbf H}
\newcommand{\spt}{\text{spt}}
\newcommand{\Sing}{\text{Sing}}
\newcommand{\R}{\text{Reg}}
\newcommand{\bd}{\partial}
\newcommand{\Cloc}{C^\infty_{\text{loc}}}
\newcommand{\N}{\textbf N}
\newcommand{\B}{\textbf B}

\title{Compactness and Rigidity of self-shrinking surfaces}
\author{Tang-Kai Lee}
\address{MIT, Dept. of Math., 77 Massachusetts Avenue, Cambridge, MA 02139-4307}
\email{tangkai@mit.edu}

\date{\today}

\begin{document}
	\begin{abstract}
	The entropy functional introduced by Colding and Minicozzi plays a fundamental role in the analysis of mean curvature flow. 
	However, unlike the hypersurface case, relatively little about the entropy is known in the higher codimensional case. 
	In this note, we use measure-theoretic techniques and rigidity results for self-shrinkers to prove a compactness theorem for a family of self-shrinking surfaces with low entropy. 
	Based on this, we prove the existence of entropy minimizers among self-shrinking surfaces and improve some rigidity results.
	\end{abstract}
	\maketitle
\section{{\bf Introduction}}
For an $n$-dimensional submanifold $\Sigma$ in $\bb R^N$ ($N>n$), the \textit{$F$-functional} (or the Gaussian area) of $\Sigma$ is defined by
\begin{equation}\label{Ffunc}
F(\Sigma) :=  (4\pi)^{-\frac n2}\int_\Sigma e^{-\frac{|x|^2}{4}}dx,
\end{equation}
and its \textit{entropy} is given by
\begin{equation}
\lambda(\Sigma):=\sup_{x_0\in\bb R^N,t_0>0} F(t_0\Sigma+x_0).
\end{equation}
The $F$-functional has been studied in various places (for example, see \cite{AIC}, \cite{H90}, and \cite{I94}). 
In \cite{CM12}, this functional was systematically investigated by Colding and Minicozzi to study the stability of the singularities of mean curvature flow (MCF) for hypersurfaces (i.e., when $N=n+1$). 
The critical points of the $F$-functional are called \textit{self-shrinkers}. 
Self-shrinkers not only lead to the simplest examples of MCF (cf. Proposition \ref{sfchar}) but also provide the singularity models for the flow (cf. \cite{H90}, \cite{I95} and \cite{W97}). 
As a result, the analyses of self-shrinkers, and thus those of the $F$-functional and the entropy functional, are crucial in the study of MCF.

One of the central topics in the study of geometric objects is the compactness problem. 
There are lots of compactness results for self-shrinkers and, more classically, minimal surfaces, such as \cite{CMcpt} and \cite{CS}. 
However, most of the compactness theorems of self-shrinkers are valid for hypersurface self-shrinkers, that is, self-shrinkers of codimension one. 
In this note, we focus on two-dimensional self-shrinkers, or self-shrinking surfaces, of any codimension. 
To this end, we consider the family $\Set_N$ of all complete, non-flat, and smooth embedded self-shrinking surfaces without boundary in $\bb R^N$ ($N>2$). 
Our main compactness theorem is stated as follows. 

\begin{thm}\label{cpt}
	Given any $\delta<2$ and $N>2,$ the space $\{\Sigma\in\Set_N:\lambda(\Sigma)\le\delta\}$ is compact in the $\Cloc$-topology.
\end{thm}

It is worth mentioning that there are plenty of example of self-shrinking surfaces, even in the hypersurface case.
In fact, unlike the one-dimensional case, where all the self-shrinking curves have been classified by Abresch and Langer \cite{AL}, it seems impossible to classify all higher dimensional self-shrinkers.
For example, any minimal submanifold of  the $n$-spheres $S^n(\sqrt{2n})$ with radius $\sqrt{2n}$ is a self-shrinker, and when $n$ is large, the $n$-sphere contains many minimal surfaces.

In the proof of Theorem \ref{cpt}, the fact that the entropy bound is strictly less than $2$ is crucial. 
This assumption prevents the situation of higher multiplicities. 
It is interesting to see if this bound could be relaxed. 
The proof of Theorem \ref{cpt} makes use of techniques in geometric measure theory, along with some rigidity results for self-shrinkers. 
We will briefly describe the proof strategies in Section \ref{1.3}.

It is expected that compactness theorems will lead to many applications. 
In this note, we use this compactness theorem to study the existence of entropy minimizers and to improve some rigidity results for self-shrinking surfaces.

\subsection{Entropy Minimizers} 
In \cite{CIMW}, Colding, Ilmanen, Minicozzi and White proved that the $n$-spheres $S^n(\sqrt{2n})$ with radius $\sqrt{2n}$ have the least entropies among all closed smooth $n$-dimensional self-shrinkers in $\bb R^{n+1}.$ 
They then conjectured that $S^n(\sqrt{2n})$ minimizes the entropy among all closed smooth hypersurfaces, which was later settled by Berstein-Wang \cite{BW} when $2\le n\le 6$ and Zhu \cite{Z} for all dimension $n\ge 2.$ 
We remark that when $n=1,$ the statement directly follows from earlier results of Gage-Hamilton \cite{GH} and Grayson \cite{G} since the entropy is monotone along the MCF.

A further natural question is whether a similar statement on entropy minimizers holds when the codimensions of self-shrinkers are higher. 
The analysis of higher codimensional MCF is a much more difficult problem than that of hypersurface MCF. 
The reasons for this scenario include the complexity of the second fundamental form and the lack of certain maximum principles. 
Consequently, many techniques and results for hypersurface MCF are hardly generalized to higher codimensional MCF.

In \cite{CM20}, Colding and Minicozzi conjectured that the only $n$-dimensional self-shrinkers of arbitrary codimensions and with entropies close to $\lambda(S^1)$ are round cylinders $S^k(\sqrt{2k})\times\bb R^{n-k}$ when $n\le 4.$ 
If this conjecture is true, then based on the codimension bounds proven in \cite{CM20}, it will imply that any ancient solution $M_t^n$ to the MCF with $n\le 4$ is a hypersurface MCF if its entropy is close to $\lambda(S^1).$ 
Note that the conjecture holds when $n=1$ since the uniqueness of ODE implies all self-shrinking curves are planar.

The first application of the compactness theorem is about the existence of entropy minimizers in the higher codimensional setting. 
To be precise, we have the following result.

\begin{thm}\label{exist}
	For any $N>2,$ there is a complete, non-flat, and smooth embedded self-shrinking surface $\Sigma$ in $\bb R^N$ without boundary such that
	$\lambda(\Sigma)\le \lambda(M)$
	for all $M\in\Set_N.$
\end{thm}

This theorem, in particular, also implies Brakke's gap theorem \cite{B78} in a concrete manner when the dimension of the self-shrinker is two. 
In other words, we established the following corollary. 

\begin{cor}\label{cor1}
  There exists $\varepsilon=\varepsilon(N)$ such that if $\Sigma$ is a smooth complete self-shrinking surface in $\bb R^{N}$ with $\lambda(\Sigma)\le 1+\varepsilon,$ then $\Sigma$ is flat.
\end{cor}

When $N=3,$ \cite{CIMW} says that the $2$-sphere $S^2(2)$ with radius $2$ is the unique entropy minimizer among closed self-shrinkers. 
At this moment, it is still not clear if $S^2(2)$ minimizes the entropy functional in $\Set_N$ for general $N>2.$ 
However, using the codimension bound in terms of the entropies of self-shrinkers established in \cite[Corollary 0.8]{CM20}, we can immediately obtain that any minimizer $\Sigma$ in Theorem \ref{exist} is contained in a Euclidean subspace of controlled dimension. 
This also tells us that the constant $\varepsilon(N)$ in Corollary \ref{cor1} does not depend on $N$ when $N$ is large enough.

\begin{thm}\label{codbd}
	Let $\Sigma$ minimize the entropy functional among all surfaces in $\Set_N.$ 
	Then $\Sigma$ is contained in a Euclidean subspace of dimension bounded by a constant $C$ independent of $N.$ 
	In particular, there exists some $N_0>0$ such that any entropy minimizer of $\Set_{N_0}$ also minimizes the entropy functional among self-shrinkers in $\Set_{N}$ for all $N\ge N_0.$
\end{thm}

\subsection{Rigidity of Self-shrinkers} The second application of the compactness result is to improve some rigidity properties of self-shrinking surfaces. 
Rigidity means that the object is isolated in the space of solutions modulo some specific geometric symmetries of the equations. 
The rigidity properties we will focus on in this note are those similar to the ``strong rigidity'' established in \cite{CIM} by Colding, Ilmanen and Minicozzi. 
They showed that any hypersurface self-shrinker that is close to a round cylinder must also be a round cylinder.
Later, Colding and Minicozzi extended the result to the higher codimensional case \cite{CM20}. 
Strong rigidity of other self-shrinkers are also investigated, such as Abresch-Langer cylinders (\cite{Z21}), Clifford tori (\cite{ELS} and \cite{SZ}), and even Ricci shrinkers (\cite{CM21}).

In this note, we establish strong rigidity properties based on different characterizations of round cylinders. 
Essentially, we will use the gap theorems in terms of the second fundamental form proven in \cite{CL} and \cite{CXZ} in the first result. 
To state our theorems, for a submanifold $\Sigma,$ we let $\h$ be its mean curvature vector, $A$ be its second fundamental form and $\mathring A$ be its traceless second fundamental form. 
Also, we let $B_R$ be the open ball of radius $R$ centered at the origin.

\begin{thm}\label{rigA}
Given $N>2,$ $\delta<2,$ and $\eta>0,$ there exist constants $R_1=R_1(N,\delta)$ and $R_2=R_2(N,\delta,\eta)$ such that the following statements hold.\\
(1) If $\Sigma$ is a self-shrinking surface in $\bb R^N$ such that $\lambda(\Sigma)\le \delta$ and $|A|^2\le \frac 12$ on $\Sigma\cap B_{R_1},$ then $\Sigma=S^k(\sqrt{2k})\times \bb R^{2-k}$ for some $k=0,1,2.$\\
(2) If $\Sigma$ is a self-shrinking surface in $\bb R^N$ such that $\lambda(\Sigma)\le \delta,$ $|\mathring A|^2\le \frac 14$ and $|\h|\ge \delta$ on $\Sigma\cap B_{R_2},$ then $\Sigma=S^k(\sqrt{2k})\times \bb R^{2-k}$ for some $k>0.$\\
\end{thm}

Recently, there are more and more global refined rigidity results for self-shrinking surfaces. 
Using those improved results, the bounds in Theorem \ref{rigA} are expected to be relaxed.	

The second rigidity result is about the characterization of being of codimension one. 
To elaborate, we would like to obtain a statement in the form that if an $n$-dimensional self-shrinker $\Sigma$ in $\bb R^N$ is a hypersurface on a large compact set (in the sense that $\Sigma\cap B_R$ is contained in an $(n+1)$-dimensional linear subspace for some large $R$), then $\Sigma$ should be a hypersurface self-shrinker. 
At this moment, using the compactness theorem and a quantification of planarity used in \cite{CM20} (cf. Proposition \ref{plane}), we are able to derive the following result for self-shrinking surfaces in $\bb R^4.$

\begin{thm}\label{rigB}
	Given $\delta<2$ and $\eta>0,$ we can find $R=R(\delta,\eta)$ such that the following statement holds. 
	If $\Sigma$ is a self-shrinking surface in $\bb R^4$ such that $\lambda(\Sigma)\le \delta$ and $\Sigma\cap B_R$ satisfies $|\h|\ge \eta$ and is contained in a $3$-dimensional linear subspace of $\bb R^4,$ then the whole $\Sigma$ is also contained in the same $3$-dimensional subspace.
\end{thm}

We remark that we require the lower bound of the mean curvature vector in Theorem \ref{rigA} (2) and Theorem \ref{rigB}. 
This guarantees the presence of the global principal normal vector $\frac{\h}{|\h|},$ which is essential in \cite{CXZ} and \cite{CM20}. 

\subsection{Organization and Proof Outline}\label{1.3}
The idea of the proof and the structure of this note are as follows. 
In Section \ref{2}, we introduce some background knowledge and notations of self-shrinkers and rectifiable varifolds. 
Later, in Section \ref{3}, we start with an arbitrary sequence in $\Set_N.$ 
Based on the entropy bound, Allard's compactness infers that we can extract a (weakly) convergent subsequence, which is done in Section \ref{3.1}. 
Then in Section \ref{3.2} and \ref{3.3}, we argue, essentially based on Allard's regularity, that the limiting varifold is smooth and that the convergence is in the strong sense. 
Using the maximum principle and the rigidity of self-shrinkers, we show that the limiting self-shrinker is non-trivial and non-flat in Section \ref{4}. 
This concludes the proof of Theorem \ref{cpt}, which implies Theorem \ref{exist}. 
Finally, we use the compactness to improve a variety of rigidity results in Section \ref{5}.

\subsection*{\bf Acknowledgement}
The author is grateful to Prof. Bill Minicozzi for his invaluable encouragement and advice. 
He would also like to thank Kai-Hsiang Wang for his comments on an earlier
draft and the referees for the constructive comments.
Part of this work was completed when the author visited the National Center for Theoretical Science in Taiwan, and he is grateful for the warm hospitality of people there.

\section{\bf Preliminaries}\label{2}
In this note, we often use $\Sigma$ or $M$ to denote a submanifold or a self-shrinker (or an $F$-stationary varifold). 
In this section, we will introduce some notations and basic properties of them.

\subsection{Characterizations of Self-Shrinkers} 
Given an $n$-dimensional submanifold $\Sigma$ in $\bb R^N,$ we call it a self-shrinker if it satisfies
\begin{equation}\label{sfsk}
\h=-\frac{x^\perp}{2}
\end{equation}
where $\h$ is the mean curvature vector of $\Sigma,$ and $x^\perp$ is the normal component of the position vector $x.$ 
We have mentioned that the critical points of the $F$-functional are self-shrinkers. 
In fact, there are several equivalent conditions of self-shrinkers. 
\begin{prop}[\cite{CM12}]\label{sfchar}
	The following statements are equivalent.
	
	(1) $\Sigma$ satisfies the self-shrinker equation \eqref{sfsk}.
	
	(2) $\sqrt{-t}\Sigma$ ($t<0$) forms a mean curvature flow. i.e., it satisfies $(\bd_t x)^\perp=\h.$
	
	(3) $\Sigma$ is a critical point of the $F$-functional defined by \eqref{Ffunc}.
	
	(4) $\Sigma$ is a minimal submanifold with respect to the metric $e^{-\frac{|x|^2}{2n}}\delta_{ij},$ where $\delta_{ij}$ is the standard metric on $\bb R^N.$
\end{prop}
\begin{proof}[Proof Sketch]
The equivalence between (1) and (2) follows from direct change-of-variable computations. 
One can see \cite[Lemma 2.2]{CM12} for a complete argument. 
The equivalence between (1) and (3) is based on the first variation formula of the $F$-functional. 
In fact, if $\Sigma_s$ ($s\in(-\varepsilon,\varepsilon)$) is a normal variation of $\Sigma$ with the variational normal vector field $V,$ then
\begin{equation}\label{Fvar}
\frac{d}{ds}F(\Sigma_s)|_{s=0}
= -(4\pi)^{-\frac n2}\int_\Sigma \pair{V,\h+\frac x2}e^{-\frac{|x|^2}{4}}dx.
\end{equation}
When $N=n+1,$ this formula is established in \cite[Lemma 3.1]{CM12}. 
The general case is similar, and one can find it in \cite{ALW}, \cite{AS} or \cite{LL}. 
Finally, (3) and (4) are equivalent based on the formula \eqref{Fvar}, the first variation formula of the volume, and the fact that the volume form of the conformal metric $e^{-\frac{|x|^2}{2n}}\delta_{ij}$ is exactly $e^{-\frac{|x|^2}{4}}dx.$
\end{proof} 

The round cylinders $S^k(\sqrt{2k})\times\bb R^{n-k}$ are the simplest $n$-dimensional self-shrinkers. 
They are also the only stable singularity models in the hypersurface case, in the sense that they are, as proven in \cite{CM12}, the only self-shrinkers that could not be perturbed away.

\subsection{Entropy and Polynomial Volume Growth} 
By definition, the entropy of a given $n$-dimensional submanifold $\Sigma$ is
\begin{equation}\label{lamexp}
\lambda(\Sigma)
= \sup_{x_0\in\bb R^N,t_0>0} F(t_0\Sigma+x_0)
=  \sup_{x_0\in\bb R^N,t_0>0} (4\pi t_0)^{-\frac n2}\int_\Sigma e^{-\frac{|x-x_0|^2}{4t_0}}dx
\end{equation}
via a change of variables. 
This expression can give us a volume control for self-shrinkers with finite entropy (cf. \cite[Section 2]{CM12}).
\begin{lemma}\label{poly}
	Let $\Sigma$ be an $n$-dimensional submanifold in $\bb R^N.$ 
	If $\lambda(\Sigma)<\infty,$ then $\Sigma$ has polynomial volume growth. 
	More precisely, for any $R>1,$ we have
	$$\V(\Sigma\cap B_R)
	\le C\cdot \lambda(\Sigma)R^n
	$$
	for some $C=C(n)>0$ where $B_R:=B_R(0)$ is the $R$-ball centered at the origin in $\bb R^N.$
\end{lemma}
\begin{proof} 
Using the expression \eqref{lamexp}, for any $R>1,$ 
$$
(4\pi R^2)^{-\frac n2} \cdot e^{-\frac 14}\V(\Sigma\cap B_R)
\le (4\pi R^2)^{-\frac n2}\int_{\Sigma\cap B_R} e^{-\frac{|x|^2}{4R^2}}
\le\lambda(\Sigma).$$
This proves the lemma with $C(n) = e^{\frac 14}\cdot(4\pi)^{\frac n2}.$ 
\end{proof}

As we mentioned, the round cylinders $S^k(\sqrt{2k})\times\bb R^{n-k}$ are the most important self-shrinkers. 
In \cite{S}, Stone calculated their entropies:
\begin{equation}\label{Sto}
\lambda(S^1)=\sqrt{\frac{2\pi}{e}}\approx 1.52
>\lambda(S^2)=\frac 4e\approx 1.47
>\lambda(S^3)>\cdots.
\end{equation}
This gives a complete list for the entropies of the round cylinders because $\lambda(\Sigma\times\bb R)=\lambda(\Sigma).$ 
In this note, we only care about those self-shrinkers with entropy less than $2$ or $\lambda(S^1)$ (when studying entropy minimizers). 
As a result, all of them have polynomial volume growth. 
Also, It is easy to see that if an immersed submanifold has entropy less than two, then it is embedded. 
Therefore, we will always assume (and obtain) embeddedness for submanifolds in our discussion.

\subsection{Varifolds and Flat Chains} 
In this note, we only discuss and deal with integer multiplicity rectifiable $n$-varifolds. 
Given an $n$-varifold $V$ in $\bb R^N,$ we write $\mu_V$ to be its corresponding measure, and $\delta V$ to be the first variation functional of $V.$ 
Its support, regular part, and singular part are denoted by $\spt V,$ $\R(V)$ and $\Sing(V).$ 
For a detailed background knowledge of varifolds, one could consult \cite{Simon}.

We say $V$ is a \textit{stationary} varifold if its first variation or its generalized mean curvature $\h$ vanishes. 
Using $\mu_V,$ we can also define the $F$-functional and the entropy functional for $V,$ in a natural manner that
$$F(V):=(4\pi)^{-\frac n2}\int_{\bb R^N} e^{-\frac{|x|^2}{4}}d\mu_V,$$
and
$$
\lambda(\Sigma)
:=  \sup_{x_0\in\bb R^N,t_0>0} (4\pi t_0)^{-\frac n2}\int_{\bb R^N} e^{-\frac{|x-x_0|^2}{4t_0}}d\mu_V.$$
Accordingly, we call $V$ an \textit{$F$-stationary} varifold if it is a critical point of the $F$-functional for varifolds. 
Equivalently, this means that its generalized mean curvature $\h$ satisfies
$$\h(x) = -\frac{x^\perp}{2}$$
for $\mu_V$-a.e. $x\in\spt V.$ 
Note that this equation makes sense since an integer multiplicity rectifiable varifold admits an approximate tangent space for $\mu_V$-a.e. $x\in\spt V.$

We also need to introduce rectifiable mod $2$ flat chains. 
The space of $n$-dimensional rectifiable mod $2$ flat chains consists of $n$-dimensional polyhedra endowed with the flat topology and $\bb Z_2$ coefficients and with some natural identifications among them. 
Given an integer multiplicity rectifiable $n$-varifold $V,$ it corresponds to an $n$-dimensional rectifiable mod $2$ flat chain $[V]$ in a natural way. 
For more details about flat chains (and their relations with varifolds), one could consult \cite{F}. 
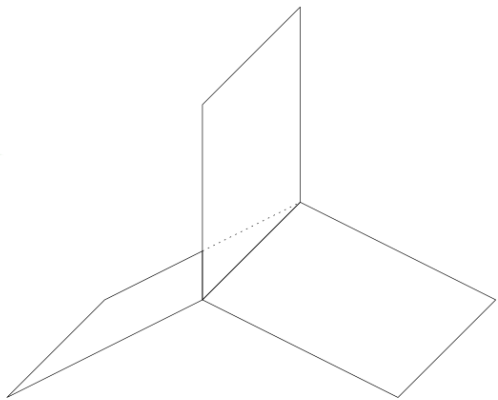
The only properties we will use are White's result (Theorem \ref{Wbd}) and the fact that three half-planes meeting at the origin with angles $2\pi/3$ form a mod $2$ flat chain $P$ with $\bd P\neq 0.$

\section{\bf Smooth Convergence}\label{3}
We start to do some preparations for the proof of Theorem \ref{cpt}. 
Recall that we let $\Set_N$ be the set of all non-flat complete self-shrinking surfaces without boundary in $\bb R^N,$ and aim at proving that the family
$$\Set_{N,\delta}:=\{M\in\Set_N:\lambda(M)\le\delta\}$$
is compact for any $\delta<2.$ 
Suppose we are given a sequence $\Sigma_i$ in $\Set_{N,\delta}.$ 
Our goal is to show that $\Sigma_i$ admits a subsequence converging to a self-shrinker in $\Set_{N,\delta}.$

\subsection{Extracting a Limit}\label{3.1}
To extract a (weakly) convergent subsequence of $\Sigma_i,$ we need Allard's compactness theorem for varifolds \cite{A}. 
The version we state here is the one in \cite{Simon}. 

\begin{thm}[Allard's compactness \cite{A}]\label{Acomp}
	Suppose $V_i$ is a sequence of integer multiplicity rectifiable $n$-varifolds in an open set $U$ of $\bb R^N.$ 
	If there exists a uniform constant $C>0$ such that
	$$\mu_{V_i}(W)+\|\delta V_i\|(W)\le C$$
	for all $i\in\bb N$ and $W\Subset U,$ then $V_i$ admits a subsequence converging to an integer multiplicity rectifiable varifold in $U.$
\end{thm}

Based on this fact, we can first extract a convergent subsequence of $\Sigma_i.$

\begin{prop}\label{varconv}
	A subsequence of $\Sigma_i$ converges to an integer multiplicity $2$-varifold $\Sigma.$
\end{prop}

\begin{proof}
Since $\Sigma_i\in\Set_{N,\delta},$ they have a uniform entropy bound. 
Lemma \ref{poly} then infers that for any $R>1,$ there is a uniform volume bound
\begin{equation}\label{Vbd}
\V(\Sigma_i\cap B_R)\le C = C(R)
\end{equation}
for any $i.$ 
On the other hand, the first variation is represented by the mean curvature when the varifold is given by a smooth submanifold. 
As a result, based on the area bound \eqref{Vbd} and the self-shrinker equation \eqref{sfsk}, we can use Allard's compactness theorem \ref{Acomp} and a standard diagonal argument to obtain a subsequence of $\Sigma_i$ which converges to a rectifiable varifold $\Sigma$ (and we still denote the subsequence as $\Sigma_i$).
\end{proof}

\subsection{Smooth Limit}\label{3.2}
Our next goal is to show that the varifold $\Sigma$ constructed in Proposition \ref{varconv} is in fact a smooth limiting space. 

\begin{prop}\label{smlim}
	The limiting varifold $\Sigma$ constructed in Proposition \ref{varconv} is smooth.
\end{prop}

To prove the smoothness of $\Sigma,$ we need Allard's regularity theorem \cite{A}. 
The version we employ here is the one stated in \cite{I95}.

\begin{thm}[Allard's regularity \cite{A}]\label{Areg}
	There exist constants $\varepsilon_1,\delta,\alpha$ and $C$ depending only on $n$ and $N$ such that the following holds. 
	Suppose $V$ is an integer multiplicity rectifiable $n$-varifold in $\bb R^N$ with locally $L^1(\mu_V)$ generalized mean curvature $\h$ and $0\in\spt V.$ 
	If 
	\begin{equation*}
	|\h|\le\frac\varepsilon r\text{ for }\mu_V\text{-a.e. }x\in B_r
	\text{ and }
	\mu_V(B_r)\le (1+\varepsilon) \omega_nr^n
	\end{equation*}
	for some $\varepsilon<\varepsilon_1$ and $r>0,$ then there is an $n$-plane $T\sbst\bb R^N$ containing $0,$ a domain $\Omega\sbst T,$ and a $C^{1,\alpha}$ function $u\colon \Omega\to T^\perp$ such that 
	\begin{equation*}
	\spt V \cap B_{\delta r} = \text{graph }u\cap B_{\delta r}
	\end{equation*}
	and
	\begin{equation*}
	\sup_\Omega \left|\frac ur\right| + \sup_\Omega |Du| + r^\alpha\sup_{x\neq y\text{ in }\Omega} \frac{|Du(x)-Du(y)|}{|x-y|^\alpha}\le C\varepsilon^{\frac{1}{4N}}.
	\end{equation*}
\end{thm}

We also need a Federer-type dimension reduction, which is often implicitly used in various dimension reduction arguments (cf. Lemma 5.8 in \cite{CIMW}).

\begin{lemma}\label{Fed}
	Let $\Gamma$ be a conical stationary rectifiable varifold in $\bb R^N.$ 
	If $y\in\Sing(\Gamma)\setminus\{0\},$ then every tangent cone to $\Gamma$ at $y$ is of the form $\Gamma'\times \bb R_y,$ where $\Gamma'$ is a conical stationary varifold in $\bb R^{N-1}$ and $\bb R_y$ is the line in the direction of $y.$
\end{lemma}

\begin{proof}[Proof of Proposition \ref{smlim}]
Based on Allard's regularity theorem \ref{Areg}, it suffices to show that any tangent cone to $\Sigma$ is a multiplicity one hyperplane, since the density of the tip $y$ of a tangent cone is the same as the density of $y$ in $\Sigma.$ 

Let $\Gamma$ be a tangent cone to $\Sigma$ at some $y\in\Sigma.$ 
Note that $\Sigma$ satisfies the shrinker equation \eqref{sfsk} in the weak sense based on the weak convergence. 
In other words, $\Sigma$ is an $F$-stationary varifold. 
Therefore, $\Sigma$ has a locally bounded generalized mean curvature, so $\Gamma$ is a stationary cone. 
As a result, $\R(\Gamma)$ is a minimal cone, and hence $\R(\Gamma)\cap S^{N-1}$ is a union of geodesics on $S^{N-1}.$ 
This implies that each connected component of $\R(\Gamma)$ is contained in a $2$-plane through the origin.

Let $P$ be a $2$-plane that intersects $\R(\Gamma).$ 
Suppose $P$ is not contained in $\R(\Gamma).$ 
Then there exists $y\in P\cap\Sing(\Gamma)\setminus\{0\}.$ Lemma \ref{Fed} then implies that any tangent cone to $\Gamma$ at $y$  is of the form $\Gamma'\times \bb R_y,$ where $\Gamma'$ is a $1$-dimensional conical stationary varifold in $\bb R^{N-1}$ and $\bb R_y$ is the line in the direction of $y.$ 
Thus $\Gamma'$ is a union of rays through the origin. 

On the other hand, the weak convergence $\Sigma_i\to \Sigma$ and the lower semicontinuity of the entropy functional implies that
\begin{equation}\label{lam<1.5}
\lambda(\Sigma)\le \liminf_{i\to\infty}\lambda(\Sigma_i)\le\delta<2.
\end{equation}
In particular, this implies that $\lambda(\Gamma')<2,$ so $\Gamma'$ consists of at most three rays through the origin since $\Gamma'$ is stationary. 
As a result, the possibilities include only three rays (triple junctions) or two rays (straight line). 
We would like to show that triple junctions could not happen as a blow-up limit. 
In fact, we will apply White's convergence result \cite[Theorem 3.3]{W09} to rule out this scenario.

\begin{thm}[\cite{W09}]\label{Wbd}
	Let $V_i$ be a sequence of integer multiplicity rectifiable varifolds that converge with locally bounded first variation to an integer multiplicity rectifiable varifold $V.$ 
	If their boundaries $\bd [V_i]$ converge (in the flat topology) to a limite chain $P,$ then $\bd[V]=P.$
\end{thm}

Both the convergence of rescaling of $\Sigma$ to $\Gamma$ and that of rescaling of $\Gamma$ to $\Gamma'\times\bb R_y$ satisfy the conditions in Theorem \ref{Wbd}. 
As a result, since $\Sigma$ has no boundary, we obtain that $\Gamma'\times \bb R_y$ also has no boundary (as a mod $2$ flat chain), so $\Gamma'$ could not be triple junctions, which implies that $\Gamma'\times \bb R_y$ could only be a $2$-plane. 
Hence, a tangent cone at $y$ is a $2$-plane, which leads to a contradiction since $y$ is a singularity. 
In conclusion, we prove that $P\sbst\R(\Gamma).$ 
The entropy bound \eqref{lam<1.5} guarantees that $P$ is the only $2$-plane that intersects $\R(\Gamma).$ 
Consequently, we derive that $\Gamma=P,$ and the smoothness of $\Sigma$ follows.
\end{proof}


\subsection{Smooth Convergence}\label{3.3}
Using Allard's regularity and the fact that $\Sigma$ is smooth (Proposition \ref{smlim}), we can improve the regularity of the convergence (which, a priori, is only a weak convergence of varifolds).
Also, a simple argument based on the maximum principle can lead to the fact that $\Sigma$ is non-trivial.

\begin{prop}\label{smconv}
	The convergence $\Sigma_i\to\Sigma$ is in the locally smooth sense, and $\Sigma$ is a non-trivial self-shrinker.
\end{prop}

To prove the non-triviality of $\Sigma,$ we observe the following simple fact of self-shrinkers.
\begin{lemma}\label{nt}
	Any $n$-dimensional complete embedded self-shrinker $M$ in $\bb R^N$ intersects the closed ball $\ovl{B_{\sqrt{2n}}}$ with radius $\sqrt{2n}.$
\end{lemma}
\begin{proof}
We will write $\n$ and $\D$ to be the Levi-Civita connection and the Laplacian operator on $M.$ 
First we prove the following simple relation. 
That is,
\begin{equation}\label{dx2}
\Delta|x|^2 = - |x^\perp|^2 + 2n.
\end{equation}
This was first established in \cite{CM12} when $N=n+1,$ based on which we give a simple proof for all codimension.
In fact, since $\D x = \h,$ the self-shrinker equation \eqref{sfsk} gives
\begin{equation*}
2\D x_i
= 2\pair{\h,\bd_i}
= -\pair{x^\perp,\bd_i}
= -x_i+  \pair{x^T,\bd_i}.
\end{equation*}
As a result,
\begin{equation*}
\D|x|^2
= 2\pair{\D x,x}+ 2|\n x|^2
= \left( -|x|^2+\pair{x^T,x}  \right) + 2n
= - |x^\perp|^2 + 2n.
\end{equation*}
Thus we verify \eqref{dx2}. 

Now we take $R>0$ such that $M\cap \ovl{B_R}\neq\emptyset.$ 
On this compact set (by the completeness and the embeddedness of $M$), we can take $p\in M\cap \ovl{B_R}$ such that
$$|p|^2 = \min_{x\in M\cap \ovl{B_R}}|x|^2.$$
Then the maximum principle implies
$$\D|x|^2(p)\ge 0\text{ and }p^T=\frac 12 \n|x|^2(p)=0.$$
Plugging these into \eqref{dx2} leads to
$$0\le -|p|^2+2n,$$
so
$$|p|\le \sqrt{2n}.$$
This concludes the proof of Lemma \ref{nt}.
\end{proof}

Now we can prove Proposition \ref{smconv} using Allard's regularity, Proposition \ref{smlim} and Lemma \ref{nt}.

\begin{proof}[Proof of Proposition \ref{smconv}]
Based on the entropy bound \eqref{lam<1.5}, the multiplicity of the convergence is $1.$ 
Since $\Sigma_i$ and $\Sigma$ are all smooth (Proposition \ref{smlim}) and have locally bounded (generalized) mean curvatures, Allard's regularity theorem \ref{Areg} gives a uniform local $C^{1,\alpha}$ graphical estimate. 
The standard elliptic PDE theory then leads to uniform higher derivative estimates, and thus the smooth convergence.

Lemma \ref{nt} infers that each $\Sigma_i$ intersects $\ovl{B_{\sqrt{2n}}},$ so the smooth convergence implies that $\Sigma\cap \ovl{B_{\sqrt{2n}}}\neq\emptyset,$ and thus $\Sigma$ is a non-trivial surface. 
\end{proof}

\section{\bf Non-flat Limits and Entropy Minimizers}\label{4}

We are now in a position to prove the main compactness result (Theorem \ref{cpt}) and the first application (Theorem \ref{exist}). 
To this end, we need a rigidity property of self-shrinkers from \cite[Corollary 6.21]{CM20}. 
Essentially, it says that if a self-shrinker is close to a round cylinder on a sufficiently large bounded set, then it must be a hypersurface in some Euclidean space.

\begin{thm}[\cite{CM20}]\label{CMrig}
	Given $k<n,$ there exists $R>2n$ such that if $M$ is an $n$-dimensional self-shrinker in $\bb R^N$ with $\lambda(M)<\infty$ and $M\cap B_R$ is the graph of a vector field $U$ over $S^k(\sqrt{2k})\times\bb R^{n-k}$ with $\|U\|_{C^1}<1/R,$ then there is an $(n+1)$-dimensional Euclidean subspace $W$ so that $M\sbst W.$
\end{thm}

Also, we need a strong rigidity result for graphical hypersurface self-shrinkers proven by Guang and Zhu \cite{GZ}. We only state it in the surface case.

\begin{thm}[{\cite[Theorem 0.1]{GZ}}]\label{Breg}
	Given $\lambda_0>0,$ there exists $R = R(\lambda_0)$ such that the following holds. 
	If $\Sigma^2\sbst \bb R^3$ is a smooth complete self-shrinker with $\lambda(\Sigma)\le\lambda_0$  and $\Sigma\cap B_R$ is graphical, then $\Sigma$ is a hyperplane.
\end{thm}

Using these rigidity properties, we can prove the main theorem. 
We state it here again.

\begin{thm}[Theorem \ref{cpt}]\label{43}
	Given any $\delta<2$ and $N>2,$ the space $\Set_{N,\delta}=\{\Sigma\in\Set_N:\lambda(\Sigma)\le\delta\}$ is compact in the $\Cloc$-topology.
\end{thm}

\begin{proof}
By Proposition \ref{smlim}, Proposition \ref{smconv} and the lower semicontinuity of the entropy, the smooth limit $\Sigma$ given by Proposition \ref{varconv} satisfies the entropy bound.
Thus, it remains to show that $\Sigma$ is not flat, which then implies that $\Sigma\in\Set_{N,\delta}.$

We use a contradiction argument. 
Suppose $\Sigma$ is flat. 
Apply Theorem \ref{CMrig} for $k=0$ and get the corresponding $R>0.$ 
On the closed ball $\ovl{B_R},$ by its compactness and the convergence in the locally smooth topology (Proposition \ref{smconv}), we may assume $\Sigma_i\cap B_R$ is the graph of a vector field $U_i$ over the planar domain $\Sigma\cap \ovl{B_R}$ with $\|U_i\|_{C^1}$ arbitrarily small for $i$ large. 
Then when $i$ is so large that $\|U_i\|_{C^1}<1/R,$ Theorem \ref{CMrig} infers that $\Sigma_i$ is contained in a $3$-plane. 
Theorem \ref{Breg} then implies that $\Sigma_i$ is a hyperplane.
Thus we derive a contradiction since we assume all $\Sigma_i$'s are non-flat. 
Hence, we prove that $\Sigma$ is not flat, so $\Sigma\in\Set_{N,\delta}.$ 
\end{proof}

As a corollary, we prove that there exists an entropy minimizer in $\Set_N$ for any $N>2.$

\begin{cor}[Theorem \ref{exist}]
	For any $N>2,$ there exists $\Sigma\in\Set_{N}$ such that
	$\lambda(\Sigma)\le \lambda(M)$
	for all $M\in\Set_N.$
\end{cor}

\begin{proof}
Take a minimizing sequence $\Sigma_i$ in $\Set_{N,\delta}$ for some $\delta\in (\lambda(S^2),2)$ such that $$\lambda(\Sigma_i)\to \inf_{M\in\Set_N}\lambda(M)$$
as $i\to\infty.$
Then Theorem \ref{cpt} implies that $\Sigma_i$ admits a subsequence converging to a non-flat self-shrinker $\Sigma\in\Set_N$ in the locally smooth topology, which gives that 
$$\lambda(\Sigma)= \inf_{M\in\Set_N}\lambda(M)$$
by the lower semicontinuity of the entropy functional again.
This proves Theorem \ref{exist}. 
\end{proof}

\section{\bf Strong Rigidity}\label{5}

In this section, we will give some other applications of our main compactness theorem.

\subsection{In Terms of Curvature}

To prove the first rigidity theorem, we will apply a gap theorem proven by Cao and Li \cite{CL}.

\begin{thm}[\cite{CL}]\label{gap1}
	Let $M$ is an $n$-dimensional complete self-shrinker without boundary and with polynomial volume growth in $\bb R^N.$ 
	If its second fundamental form satisfies
	\footnote{Note that our definition of self-shrinkers is different from that in \cite{CL} up to a constant, so the bound also differs from a multiple $2.$}
	$$|A|^2\le \frac 12,
	$$
	then $M$ is $S^k(\sqrt{2k})\times R^{n-k}$ for some $k=0,\cdots,n.$
\end{thm}

This result was first proven by Le and Sesum \cite{LS} for the hypersurface case, and then generalized to the full generality by Cao and Li later. 
Now, we are going to prove the first rigidity result using this gap theorem and the main compactness theorem.

\begin{prop}
	Given $\delta<2,$ there exists $R=R(N,\delta)>0$ such that the following statement holds. 
	If $\Sigma$ is a smooth complete self-shrinking surface in $\bb R^N$ with $\lambda(\Sigma)\le \delta$ and its second fundamental form satisfies $|A|^2\le \frac 12$ on $B_R,$ then $\Sigma=S^k(\sqrt{2k})\times\bb R^{2-k}$ for some $k=0,1,2.$
\end{prop}

\begin{proof}
	We will prove the statement by contradiction. 
	If the statement were wrong, then for any $R>0,$ there would exist a smooth complete self-shrinking surface $\Sigma_R$ in $\bb R^N$ with $\lambda(\Sigma_R) \le \delta$ such that $\Sigma_R$ is not a round cylinder but its second fundamental form satisfies $|A|^2\le \frac 12$ on $\Sigma_R\cap B_R.$ 
	Then, Theorem \ref{cpt} implies that we can find a sequence $R_i\to\infty$ as $i\to\infty$ such that the sequence $\Sigma_{R_i}$ converges to a non-flat limit $\Sigma$ in the $\Cloc$-sense. 
	
	We can use Theorem \ref{CMrig} and the same argument in the proof of Theorem \ref{43} to show that this limit is not a round cylinder. 
	However, the smooth convergence implies that $|A|^2\le 1/2$ on $\Sigma.$ 
	As a result, we know that $\Sigma$ is a round cylinder by the gap theorem \ref{gap1}. 
	Therefore, we derive a contradiction.
\end{proof}

This proves the first conclusion of Theorem \ref{rigA}. 
Using exactly the same argument, we can prove the second conclusion of Theorem \ref{rigA} based on \cite[Theorem 1.6]{CXZ}. 
Note that we need a uniform lower bound of the magnitude of the mean curvature here since in the result of \cite{CXZ}, they focused on self-shrinkers with non-vanishing mean curvature (so that the principal normal is well-defined and they could proceed their arguments).

\subsection{In Terms of Planarity}

Next, we discuss self-shrinking surfaces that are of codimension one on large compact sets. 
To this end, we need to ``quantify'' the notion of planarity. 
This was done in \cite{CM20} for those surfaces in $\bb R^4$ with non-vanishing mean curvature vector.\footnote{For the details, see Section 12.2 in \cite{CM20}.} 

\begin{definition}[\cite{CM20}]
	Given a surface $\Sigma^2\sbst\bb R^4$ with non-vanishing mean curvature $\h,$ let $\N:=\frac{\h}{|\h|}$ be its principal normal. 
	Let $J$ be the almost complex structure of the normal bundle $N\Sigma,$ and we define the binormal $\B:=J\N.$
\end{definition} 

An important property of this binormal is the following fact, which can be proven via a system of Frenet-Serret type equations and a general derivative equation of $\h.$ 

\begin{prop}[\cite{CM20}]\label{plane}
	Given a self-shrinking surface $\Sigma^2\sbst\bb R^4$ with non-vanishing mean curvature, it is contained in a $3$-dimensional hyperplane if and only if $\pair{A,\B}=0$ on $\Sigma.$
\end{prop}

Using this characterization, we are able to give a rigidity result in terms of the planarity for the self-shrinking surfaces in $\bb R^4.$ 
In fact, we will require a mean curvature lower bound and use the planarity in a large compact set to show that the self-shrinking surface must be round cylinders. In particular, they are of codimension one.

\begin{thm}[Theorem \ref{rigB}]
	Given $\delta<2$ and $\eta>0,$ we can find $R=R(\delta,\eta)$ such that the following statement holds. 
	If $\Sigma$ is a self-shrinking surface in $\bb R^4$ such that $\lambda(\Sigma)\le \delta$ and $\Sigma\cap B_R$ satisfies $|\h|\ge \eta$ and is contained in a $3$-dimensional linear subspace of $\bb R^4,$ then the whole $\Sigma$ is also contained in the same $3$-dimensional subspace.
\end{thm}

\begin{proof}
We will prove the statement by contradiction. 
If the statement were wrong for any $R>0,$ then for any $R>0,$ we can find a self-shrinking surface $\Sigma_R$ in $\bb R^N$  with $\lambda(\Sigma_R)\le \delta$ such that $\Sigma_R\cap B_R$ has $|\h|\ge\eta$ and $\pair{A,\B}=0$ but $\Sigma_R$ is not contained in any $3$-dimensional linear subspace. 
(That is, $\Sigma_R$ is not a hypersurface.) 
Then, Theorem \ref{cpt} implies that there is a sequence $R_i\to\infty$ as $i\to\infty$ such that the sequence $\Sigma_{R_i}$ converges to a non-flat limit $\Sigma$ in the $\Cloc$-sense. 

Using Proposition \ref{plane} again, we know that $\Sigma$ is a hypersurface, and also it satisfies $|\h|\ge\eta.$ 
In particular, it is a mean convex self-shrinking surface of codimension one. 
Therefore, by the classification of mean convex self-shrinker (cf. \cite{H90} and \cite{CM12}), $\Sigma$ must be a round cylinder. 
However, by the rigidity of round cylinders (Theorem \ref{CMrig}), this implies that for $i$ large enough, $\Sigma_{R_i}$ is also a hypersurface, contradicting the fact that $\Sigma_{R_i}$ is not contained in any $3$-plane. 
\end{proof}


\begin{thebibliography}{999999}
	
	\bib{AL}{article}{
		author={Abresch, U.},
		author={Langer, J.},
		title={The normalized curve shortening flow and homothetic solutions},
		journal={J. Differential Geom.},
		volume={23},
		date={1986},
		number={2},
		pages={175--196},
		issn={0022-040X},
		review={\MR{845704}},
	}

	\bib{A}{article}{
		author={Allard, W. K.},
		title={On the first variation of a varifold},
		journal={Ann. of Math. (2)},
		volume={95},
		date={1972},
		pages={417--491},
		issn={0003-486X},
		review={\MR{307015}},
	}

	\bib{ALW}{article}{
		author={Andrews, B.},
		author={Li, H.},
		author={Wei, Y.},
		title={$\scr F$-stability for self-shrinking solutions to mean curvature
			flow},
		journal={Asian J. Math.},
		volume={18},
		date={2014},
		number={5},
		pages={757--777},
		issn={1093-6106},
		review={\MR{3287002}},
	}

	\bib{AIC}{article}{
		author={Angenent, S.},
		author={Ilmanen, T.},
		author={Chopp, D. L.},
		title={A computed example of nonuniqueness of mean curvature flow in
			$\bold R^3$},
		journal={Comm. Partial Differential Equations},
		volume={20},
		date={1995},
		number={11-12},
		pages={1937--1958},
		issn={0360-5302},
		review={\MR{1361726}},
	}

	\bib{AS}{article}{
		author={Arezzo, C.},
		author={Sun, J.},
		title={Self-shrinkers for the mean curvature flow in arbitrary
			codimension},
		journal={Math. Z.},
		volume={274},
		date={2013},
		number={3-4},
		pages={993--1027},
		issn={0025-5874},
		review={\MR{3078255}},
	}

	\bib{BW}{article}{
	author={Bernstein, J.},
	author={Wang, L.},
	title={A sharp lower bound for the entropy of closed hypersurfaces up to
		dimension six},
	journal={Invent. Math.},
	volume={206},
	date={2016},
	number={3},
	pages={601--627},
	issn={0020-9910},
	review={\MR{3573969}},
}
	
	\bib{B78}{book}{
		author={Brakke, K. A.},
		title={The motion of a surface by its mean curvature},
		series={Mathematical Notes},
		volume={20},
		publisher={Princeton University Press, Princeton, N.J.},
		date={1978},
		pages={i+252},
		isbn={0-691-08204-9},
		review={\MR{485012}},
	}

	\bib{CL}{article}{
		author={Cao, H.-D.},
		author={Li, H.},
		title={A gap theorem for self-shrinkers of the mean curvature flow in
			arbitrary codimension},
		journal={Calc. Var. Partial Differential Equations},
		volume={46},
		date={2013},
		number={3-4},
		pages={879--889},
		issn={0944-2669},
		review={\MR{3018176}},
	}

	\bib{CXZ}{article}{
	author={Cao, S. J.},
	author={Xu, H. W.},
	author={Zhao, E. T.},
	title={Pinching theorems for self-shrinkers of higher codimension},
	note={preprint available at \url{http://www.cms.zju.edu.cn/upload/file/20170320/1489994331903839.pdf}},
	date={2014}
}

	\bib{CIM}{article}{
		author={Colding, T. H.},
		author={Ilmanen, T.},
		author={Minicozzi, W. P., II},
		title={Rigidity of generic singularities of mean curvature flow},
		journal={Publ. Math. Inst. Hautes \'{E}tudes Sci.},
		volume={121},
		date={2015},
		pages={363--382},
		issn={0073-8301},
		review={\MR{3349836}},
	}

	\bib{CIMW}{article}{
	author={Colding, T. H.},
	author={Ilmanen, T.},
	author={Minicozzi, W. P., II},
	author={White, B.},
	title={The round sphere minimizes entropy among closed self-shrinkers},
	journal={J. Differential Geom.},
	volume={95},
	date={2013},
	number={1},
	pages={53--69},
	issn={0022-040X},
	review={\MR{3128979}},
}
	
	\bib{CM11}{book}{
		author={Colding, T. H.},
		author={Minicozzi, W. P., II},
		title={A course in minimal surfaces},
		series={Graduate Studies in Mathematics},
		volume={121},
		publisher={American Mathematical Society, Providence, RI},
		date={2011},
		pages={xii+313},
		isbn={978-0-8218-5323-8},
		review={\MR{2780140}},
	}
	
	\bib{CM12}{article}{
		author={Colding, T. H.},
		author={Minicozzi, W. P., II},
		title={Generic mean curvature flow I: generic singularities},
		journal={Ann. of Math. (2)},
		volume={175},
		date={2012},
		number={2},
		pages={755--833},
		issn={0003-486X},
		review={\MR{2993752}},
	}

	\bib{CMcpt}{article}{
		author={Colding, T. H.},
		author={Minicozzi, W. P., II},
		title={Smooth compactness of self-shrinkers},
		journal={Comment. Math. Helv.},
		volume={87},
		date={2012},
		number={2},
		pages={463--475},
		issn={0010-2571},
		review={\MR{2914856}},
	}

	\bib{CM20}{article}{
		author={Colding, T. H.},
		author={Minicozzi, W. P., II},
		title={Complexity of parabolic systems},
		journal={Publ. Math. Inst. Hautes \'{E}tudes Sci.},
		volume={132},
		date={2020},
		pages={83--135},
		issn={0073-8301},
		review={\MR{4179832}},
	}

	\bib{CM21}{article}{
		author={Colding, T. H.},
		author={Minicozzi, W. P., II},
		title={Singularities of Ricci flow and diffeomorphisms},
		note={Preprint available at \url{https://arxiv.org/abs/2109.06240}},
		date={2021}
	}

	\bib{CS}{article}{
		author={Choi, H. I.},
		author={Schoen, R.},
		title={The space of minimal embeddings of a surface into a
			three-dimensional manifold of positive Ricci curvature},
		journal={Invent. Math.},
		volume={81},
		date={1985},
		number={3},
		pages={387--394},
		issn={0020-9910},
		review={\MR{807063}},
	}

	\bib{ELS}{article}{
		author={Evans, C. G.},
		author={Lotay, J. D.},
		author={Schulze, F.},
		title={Remarks on the self-shrinking Clifford torus},
		journal={J. Reine Angew. Math.},
		volume={765},
		date={2020},
		pages={139--170},
		issn={0075-4102},
		review={\MR{4129358}},
	}

	\bib{F}{book}{
		author={Federer, H.},
		title={Geometric measure theory},
		series={Die Grundlehren der mathematischen Wissenschaften, Band 153},
		publisher={Springer-Verlag New York Inc., New York},
		date={1969},
		pages={xiv+676},
		review={\MR{0257325}},
	}

	\bib{GH}{article}{
		author={Gage, M.},
		author={Hamilton, R. S.},
		title={The heat equation shrinking convex plane curves},
		journal={J. Differential Geom.},
		volume={23},
		date={1986},
		number={1},
		pages={69--96},
		issn={0022-040X},
		review={\MR{840401}},
	}
	
	\bib{G}{article}{
		author={Grayson, M. A.},
		title={The heat equation shrinks embedded plane curves to round points},
		journal={J. Differential Geom.},
		volume={26},
		date={1987},
		number={2},
		pages={285--314},
		issn={0022-040X},
		review={\MR{906392}},
	}

	\bib{GZ}{article}{
		author={Guang, Q.},
		author={Zhu, J. J.},
		title={Rigidity and curvature estimates for graphical self-shrinkers},
		journal={Calc. Var. Partial Differential Equations},
		volume={56},
		date={2017},
		number={6},
		pages={Paper No. 176, 18},
		issn={0944-2669},
		review={\MR{3721649}},
	}
	
	\bib{H90}{article}{
		author={Huisken, G.},
		title={Asymptotic behavior for singularities of the mean curvature flow},
		journal={J. Differential Geom.},
		volume={31},
		date={1990},
		number={1},
		pages={285--299},
		issn={0022-040X},
		review={\MR{1030675}},
	}
	
	\bib{H93}{article}{
		author={Huisken, G.},
		title={Local and global behaviour of hypersurfaces moving by mean
			curvature},
		conference={
			title={Differential geometry: partial differential equations on
				manifolds},
			address={Los Angeles, CA},
			date={1990},
		},
		book={
			series={Proc. Sympos. Pure Math.},
			volume={54},
			publisher={Amer. Math. Soc., Providence, RI},
		},
		date={1993},
		pages={175--191},
		review={\MR{1216584}},
	}
	
	\bib{I94}{article}{
		author={Ilmanen, T.},
		title={Elliptic regularization and partial regularity for motion by mean
			curvature},
		journal={Mem. Amer. Math. Soc.},
		volume={108},
		date={1994},
		number={520},
		pages={x+90},
		issn={0065-9266},
		review={\MR{1196160}},
	}
	
	\bib{I95}{article}{
		author={Ilmanen, T.},
		title={Singularities of Mean Curvature Flow of Surfaces},
		note={Preprint available at \url{http://www.math.ethz.ch/~ilmanen/papers/pub.html}},
		date={1995}
	}

	\bib{LS}{article}{
		author={Le, N. Q.},
		author={Sesum, N.},
		title={Blow-up rate of the mean curvature during the mean curvature flow
			and a gap theorem for self-shrinkers},
		journal={Comm. Anal. Geom.},
		volume={19},
		date={2011},
		number={4},
		pages={633--659},
		issn={1019-8385},
		review={\MR{2880211}},
	}

	\bib{LL}{article}{
		author={Lee, Y.-I.},
		author={Lue, Y.-K.},
		title={The stability of self-shrinkers of mean curvature flow in higher
			co-dimension},
		journal={Trans. Amer. Math. Soc.},
		volume={367},
		date={2015},
		number={4},
		pages={2411--2435},
		issn={0002-9947},
		review={\MR{3301868}},
	}

	\bib{Simon}{book}{
		author={Simon, L.},
		title={Lectures on geometric measure theory},
		series={Proceedings of the Centre for Mathematical Analysis, Australian
			National University},
		volume={3},
		publisher={Australian National University, Centre for Mathematical
			Analysis, Canberra},
		date={1983},
		pages={vii+272},
		isbn={0-86784-429-9},
		review={\MR{756417}},
	}

	\bib{S}{article}{
		author={Stone, A.},
		title={A density function and the structure of singularities of the mean
			curvature flow},
		journal={Calc. Var. Partial Differential Equations},
		volume={2},
		date={1994},
		number={4},
		pages={443--480},
		issn={0944-2669},
		review={\MR{1383918}},
	}

	\bib{SZ}{article}{
		author={Sun, A.},
		author={Zhu, J. J.}
		title={Rigidity and Łojasiewicz inequalities for Clifford self-shrinkers},
		note={preprint available at \url{https://arxiv.org/abs/2011.01636}},
		date={2020}
	}
	
	\bib{W97}{article}{
		author={White, B.},
		title={Partial regularity of mean-convex hypersurfaces flowing by mean
			curvature},
		journal={Internat. Math. Res. Notices},
		date={1994},
		number={4},
		pages={185-192},
			issn={1073-7928},
			review={\MR{1266114}},
		}

	\bib{W05}{article}{
		author={White, B.},
		title={A local regularity theorem for mean curvature flow},
		journal={Ann. of Math. (2)},
		volume={161},
		date={2005},
		number={3},
		pages={1487--1519},
		issn={0003-486X},
		review={\MR{2180405}},
	}

	\bib{W09}{article}{
		author={White, B.},
		title={Currents and flat chains associated to varifolds, with an
			application to mean curvature flow},
		journal={Duke Math. J.},
		volume={148},
		date={2009},
		number={1},
		pages={41--62},
		issn={0012-7094},
		review={\MR{2515099}},
	}

	\bib{Z}{article}{
		author={Zhu, J. J.},
		title={On the entropy of closed hypersurfaces and singular
			self-shrinkers},
		journal={J. Differential Geom.},
		volume={114},
		date={2020},
		number={3},
		pages={551--593},
		issn={0022-040X},
		review={\MR{4072205}},
	}

	\bib{Z21}{article}{
		author={Zhu, J. J.},
		title={\L ojasiewicz inequalities, uniqueness and rigidity for cylindrical self-shrinkers},
		note={Preprint available at \url{https://arxiv.org/abs/2011.01633}},
		date={2020}
	}
\end{thebibliography}
\end{document}